\documentclass[11pt]{amsart}
\usepackage{amsmath}
\usepackage{amssymb}
\usepackage{amsthm}

\usepackage{fullpage}

\newtheorem{theorem}{Theorem}[section]
\newtheorem{lemma}[theorem]{Lemma}
\newtheorem{proposition}[theorem]{Proposition}
\newtheorem{question}[theorem]{Question}
\newtheorem{corollary}[theorem]{Corollary}

\theoremstyle{remark}
\newtheorem*{remarks}{Remarks}
\newtheorem*{remark}{Remark}

\newcommand{\Vol}{\text{Vol}}
\newcommand{\dVol}{\text{dVol}}
\newcommand{\tr}{\text{tr}}
\newcommand{\Sym}{\text{Sym}}
\newcommand{\Ric}{\text{Ric}}

\newcommand{\injrad}{\text{injrad}}
\newcommand{\diam}{\text{diam}}

\title{Conformal classes realizing the Yamabe invariant}
\author{Heather Macbeth}
\address{Department of Mathematics, Princeton University; Fine Hall, Washington Rd, Princeton, NJ 08544}
\email{macbeth@math.princeton.edu}

\begin{document}

\maketitle

\begin{abstract}
We give a characterization of conformal classes realizing a compact manifold's Yamabe invariant.  This characterization is the analogue of an observation of Nadirashvili for metrics realizing the maximal first eigenvalue, and of Fraser and Schoen for metrics realizing the maximal first Steklov eigenvalue.
\end{abstract}

\section{Introduction}

\subsection{The result}

Three decades ago, Schoen's ground-breaking solution \cite{schoen84} of the Yamabe problem established that, within any conformal equivalence class $c$ of Riemannian metrics on a compact smooth $n$-manifold $M$, the total scalar curvature functional
\[
g\mapsto\frac{\int_M R(g)\dVol_g}{\left(\int_M \dVol_g\right)^{1-2/n}}
\]
attains its infimum.  The quantity
\[
I(c):=\min_{g\in c}\frac{\int_M R(g)\dVol_g}{\left(\int_M \dVol_g\right)^{1-2/n}},
\]
which is therefore well-defined, is known as the the \emph{Yamabe constant} of $c$, and metrics $g$ attaining this minimum are referred to as \emph{Yamabe metrics}.  Computing the Euler-Lagrange equation of the restriction of the total scalar curvature functional to $c$ shows that Yamabe metrics have constant scalar curvature.

One can study the properties of $I(c)$ as it varies over all conformal classes $c$ on a smooth manifold.  In particular the \emph{Yamabe invariant} of $M$ is defined to be the minimax expression
\[
Y(M):=\sup_cI(c)=\sup_c\min_{g\in c}\frac{\int_M R(g)\dVol_g}{\left(\int_M \dVol_g\right)^{1-2/n}};
\]
This quantity is finite, as follows from the observation of Aubin \cite{aub76b} that
$I(c)\leq n(n-1)\omega_n^{2/n}$  for all $c$
(it is a corollary of the solution of the Yamabe problem that equality holds if and only if $c$ is the conformal class of the round sphere). \newline

In this paper we study manifolds $M$ whose Yamabe invariant $Y(M)$ is attained.  Our result is the following algebraic constraint on the set of Yamabe metrics in a conformal class attaining the Yamabe invariant:

\begin{theorem} \label{conv-hull}
Let $M$ be a compact smooth $n$-manifold, $c$ a conformal class attaining  $Y(M)$, and $\mathcal{F}$ the set of unit-volume Yamabe metrics in $c$.  There exists a (positive)  measure $\mu$ on $\mathcal{F}$, such~that
\[
0=\int_{\mathcal{F}}  \left(\Ric(g)-\tfrac{1}{n}R(g)g\right)|\dVol_{g}|^{1-2/n}d\mu(g).
\]
\end{theorem}

\begin{remarks} \begin{enumerate} \item Here $|\dVol_{g}|$ is the density associated to the volume form $\dVol_{g}$; this density is a positive section of the density bundle $|\Lambda^nM|$, an oriented line bundle, and so $|\dVol_{g}|^{1-2/n}$ is a well-defined section of $|\Lambda^nM|^{1-2/n}$.
\item The integral is in the sense of Pettis; see Subsection \ref{functional}.
 \item More explicitly:  let $g$ be a representative of $c$;  and $\mathcal{U}$ the set of positive smooth functions $u$ such that $(u^2g)$ is a unit-volume Yamabe metrics in $c$; then there exists a measure $\mu$ on $\mathcal{U}$, such that
\[
0=\int_{\mathcal{U}}    u^{n-2}\left[\Ric(u^2g)-\tfrac{1}{n}R(u^2g)u^2g\right]d\mu(u).
\]
\end{enumerate}
\end{remarks}

\begin{corollary} \label{conv-hull-finite}
Let $M$ be a compact smooth $n$-manifold, and $c$ a conformal class attaining $Y(M)$, and suppose that there are, up to rescaling, only finitely many Yamabe metrics in $c$.  Then there exist $N\in\mathbb{N}$, and Yamabe metrics  $(g_i)_{1\leq i\leq N}$ (not necessarily unit-volume) in $c$, such that
\begin{equation*}
0=\sum_{i=1}^N  \left(\Ric(g_i)-\tfrac{1}{n}R(g_i)g_i\right)|\dVol_{g_i}|^{1-2/n}.
\end{equation*}
\end{corollary}

\begin{proof} Write $\mathcal{F}$ explicitly as  $(\overline{g}_{j})$. By Theorem \ref{conv-hull}  there exist nonnegative reals $(a_{j})$ such that $0=\sum_j a_j \ \left(\Ric(\overline{g}_j)-\tfrac{1}{n}R(\overline{g}_j)\overline{g}_j\right)|\dVol_{\overline{g}_j}|^{1-2/n}$.
Drop those  $\overline{g}_j$ for which $a_{j}=0$, and rescale~the~rest.
\end{proof}

The interest of Theorem \ref{conv-hull} and Corollary \ref{conv-hull-finite} lies in their connection to a subtle and very appealing open problem:
\begin{question} \label{is-einstein}
  Do all conformal classes attaining the Yamabe invariant contain an Einstein metric?
\end{question}
  This question is motivated by the examples of manifolds for which the Yamabe invariant has to date been computed, which fall into two classes:
\begin{itemize}
\item Einstein manifolds $(M,g)$ for which it has been proved that $[g]$, the conformal class of the Einstein metric, attains $Y(M)$; these include
    \begin{itemize}
    \item $S^n$ (\cite{aub76a}; see \cite[Section 3]{PL}  for an exposition);
\item manifolds admitting flat metrics (\cite{SY79,GL83}; see \cite[Proposition 1.3]{schoen87}  for an exposition);
    \item general-type and Calabi-Yau complex surfaces \cite{lebrun96,lebrun-survey};
    \item $\mathbb{CP}^2$ \cite{lebrun97};
    \item $\mathbb{RP}^3$ \cite{BN04};
    \end{itemize}
\item manifolds $M$ for which the existence of a conformal class attaining $Y(M)$ is unclear, or for which it has been proved that no conformal class attains $Y(M)$; these include
    \begin{itemize}
    \item $S^1\times S^{n-1}$ and connect sums thereof \cite{kob87};
    \item $M\#k\overline{\mathbb{CP}^2}$, where $M$ is a complex surface of general type \cite{lebrun96};
    \item $M\#[S^1\times S^3]$, where $M$ is a 4-manifold with $Y(M)\leq 0$ \cite{petean};
    \item $\mathbb{RP}^3\#k(S^1\times\mathbb{RP}^2)$ \cite{BN04};
    \end{itemize}
\end{itemize}
as well as by the observation that the Yamabe invariant is a minimax quantity associated with the total scalar curvature functional
\[
g\mapsto\frac{\int_M R(g)\dVol_g}{\left(\int_M \dVol_g\right)^{1-2/n}},
\]
of which Einstein metrics are the critical points.   If Question \ref{is-einstein} were answered in the affirmative, it could perhaps be possible to find Einstein metrics on new manifolds by direct variational methods, by maximizing the functional $I$.

Theorem \ref{conv-hull} generalizes two previous results on Question \ref{is-einstein}.  In the case when the finite set of Theorem \ref{conv-hull} consists simply of a single metric $g$, the following well-known  result is obtained:
\begin{lemma}[see for example  {\cite[discussion preceding Lemma 1.2]{schoen87}}] \label{unique}
Suppose that the conformal class $c$ attains $Y(M)$, and suppose that there is, up to rescaling, exactly one Yamabe metric $g$ in $c$.  Then $g$ is Einstein.
\end{lemma}

\begin{proof}[Proof, given Theorem \ref{conv-hull}]
The tensor $\Ric(g)-\tfrac{1}{n}R(g)g$ vanishes precisely when $g$ is Einstein.
\end{proof}

In the case when the class $c$ contains exactly two unit-volume Yamabe metrics (i.e., $|\mathcal{F}|=2$), we may write them explicitly as $g$ and $u^2g$, and obtain the following relationship, previously derived by Anderson \cite[equation 2.37]{and}:
\[
u^{-1}(1+u^{n-2})[\Ric(g)-\tfrac{1}{n}R(g)g]=-(n-2)[\text{Hess}_g(u^{-1})-\tfrac{1}{n}\Delta_g(u^{-1})].
\]
Anderson then gives a Bianchi-identity argument deducing a contradiction from this relationship unless the two metrics coincide, thus answering Question \ref{is-einstein} in the affirmative in this case:
\begin{theorem}[{\cite[Theorem 1.2]{and}}] \label{anderson}
Suppose that the conformal class $c$ attains $Y(M)$, and suppose that there are at most two (modulo rescaling) Yamabe metrics in $c$.  Then in fact there is exactly one Yamabe metric, and that metric is Einstein.
\end{theorem}

We note a different point of view, from which the available evidence regarding Question \ref{is-einstein} is more equivocal.  There is a bidirectional version of Lemma \ref{unique} which rephrases Question \ref{is-einstein} as a question of uniqueness:
\begin{proposition}[\cite{obata}, see also {\cite[Proposition 1.4]{schoen87}} ]
Let $M$ be a compact manifold other than the sphere, and suppose the conformal class $c$ attains $Y(M)$.  Then there is an Einstein metric in $c$ if and only there is a unique (modulo rescaling) Yamabe metric in $c$;  the Einstein metric is the Yamabe metric if so.
\end{proposition}

In a general (that is, not necessarily maximizing) conformal class, the question of uniqueness (modulo rescaling) of Yamabe metrics is well-studied. Let $\Omega$ be the subset of conformal classes containing a unique Yamabe metric.  $\Omega$ is large (for instance it contains all conformal classes $c$ such that $I(c)\leq 0$ \cite{aub70}, and it is open and dense \cite{and}) but its complement need not actually be empty; the simplest counterexample is the round sphere, and many others have been found \cite{schoen87,BP}.\newline

We therefore hope that future work may use the results of this paper to provide a resolution of Question \ref{is-einstein}, in one of two ways.  On the one hand, there may exist a Bianchi-type argument generalizing Anderson's, which uses Theorem \ref{conv-hull} to answer the question in the affirmative, at least when $\mathcal{F}$ is finite.  On the other hand, it may in fact be possible to construct a set of metrics satisfying the algebraic criterion of Theorem \ref{conv-hull}, thus answering the question in the negative.

\subsection{The technique}

This paper was inspired by two closely analogous results in other geometric minimax contexts.  Nadirashvili \cite{nadir} studies Riemannian metrics $g$ on a compact smooth manifold $M$ which maximize the weighted first Dirichlet eigenvalue $\lambda_1(g)\Vol(M, g)^{2/n}$, where
\[
\lambda_1(g) = \inf_{\substack{\{u\in\mathcal{C}^\infty(M):\\0= \ \int u\}}}\frac{\int_M |\nabla u|^2\dVol_g}{\int_M u^2\dVol_g}.
\]
 Among other things, he observes (Theorem 5) that such metrics are \emph{$\lambda_1$-minimal}; that is, that there exist a set of first Dirichlet eigenfunctions for $g$, such that the map of $M$ into Euclidean space defined by those eigenfunctions is an isometric embedding as a minimal submanifold of the unit sphere.  The key point of the proof is to find a finite set $(u_i)$ of first eigenfunctions, such that
\begin{equation*}
g=\sum_i  \left[(du_i)^{\otimes 2} + \frac{1}{4}\Delta_g(u_i^2)g\right].
\end{equation*}

Fraser and Schoen \cite{FS} study Riemannian metrics $g$ on a compact smooth manifold with boundary $M$ which maximize the weighted first Steklov eigenvalue $\sigma_1(g) \Vol(\partial M, g|_{\partial M})^{1/(n-1)}$, where
\[
\sigma_1(g)
=\inf_{\substack{\{u\in\mathcal{C}^0(M):\\u\text{ harmonic},\\ \int_{\partial M}u=0\}}}
\frac{\int_ M |\nabla u|^2\dVol_g}{\int_{\partial M} u^2\dVol_{g|_{\partial M}}}.
\]
Among other things, they observe (Proposition 5.2) that such a metric $g$ has a set of first Steklov eigenfunctions, such that the map of $M$ into Euclidean space defined by those eigenfunctions is a conformal embedding as a minimal submanifold of the unit ball, isometric on $\partial M$.  The key point of the proof is to find a finite set $(u_i)$ of first eigenfunctions, such that
\begin{eqnarray*}
0&=&\sum_i \left[(du_i)^{\otimes 2} - \frac{1}{2}|du_i|_g^2g\right],\quad\text{on $M$},\\
1&=&\sum_i  u_i^2,\quad\text{on $\partial M$}.
\end{eqnarray*}

Our argument, particularly in Section \ref{sect-hahn-banach}, closely follows Fraser and Schoen's.

\subsection{Outline}

The organization of this paper is as follows.  Section \ref{sect-compactness} reviews some preliminaries.  Section \ref{sect-euler-lagrange} is dedicated to the proof of Proposition \ref{euler-lagrange}; this proposition establishes control on the formal derivative of the Yamabe functional at a conformal class which attains the Yamabe invariant.  A Hahn-Banach theorem argument in Section \ref{sect-hahn-banach} completes the proof of Theorem \ref{conv-hull}.

\subsection{Acknowledgements}

Discussions with Rod Gover, Richard Schoen and my advisor Gang Tian have helped shape this paper, and I am grateful to them for their interest.  I also especially thank Georgios Moschidis, who drew my attention to an error in the main theorem of this paper as it appeared in an earlier version of the paper (see remark following Proposition \ref{conv-hull-exp}).

\section{Preliminaries} \label{sect-compactness}

\subsection{Bounds on constant-scalar-curvature metrics}

We give a version of the standard bound on metrics with constant scalar curvature, in which the dependence of the bounds on the metric is made explicit.

By convention we define
\[
\Lambda := 4n(n-1)\left(\frac{\Gamma(\tfrac{n}{2})\Gamma(\tfrac{n}{2}+1)\omega_{n-1}}{\Gamma(n+1)}\right)^{2/n}=n(n-1)\omega_n^{\frac{2}{n}}.
\]
Here $\omega_n$ is the volume of the $n$-sphere.

In a fixed conformal class, the standard a priori bounds on metrics with constant scalar curvature are essentially equivalent to the solution to the Yamabe problem for $(M,[g])$ with $I([g])<\Lambda$, and are due to Trudinger \cite{tru68} and Aubin \cite{aub76a,aub76b} (see also the exposition in \cite{PL}):

\begin{theorem} \label{higher-order-single}
Suppose given $m\geq 2$, $p>n$, $\eta$, and a smooth metric $g$ on $M$.  There exists $C>0$ dependent only on $n$, $m$, $p$, $\eta$, $g$, such that for each smooth positive function  $\varphi$ on $M$ such that $\varphi^{4/(n-2)}g$ has volume 1 and constant scalar curvature $\lambda \leq\Lambda-\eta$,
we have the uniform bound
\[
||\varphi||_{\mathcal{W}^{m,p},g}\leq C.
\]
\end{theorem}

This theorem holds more generally for a compact set of conformal classes:

\begin{theorem} \label{higher-order}
Suppose given $m\geq 2$, $p>n$, $i$, $D$, $B$, $\eta$.  There exists $C>0$ dependent only on $n$, $m$, $p$, $i$, $D$, $B$, $\eta$, such that for each
\begin{itemize}
\item smooth metric $g$ on $M$ such that $\diam(M,g)\leq D$, $\injrad(M,g)\geq i$, $||\Ric(g)||_{\mathcal{C}^{m-2},g}\leq B$,
\item smooth positive function  $\varphi$ on $M$ such that $\varphi^{4/(n-2)}g$ has volume 1 and constant scalar curvature $\lambda \leq\Lambda-\eta$,
\end{itemize}
we have the uniform bound
\[
||\varphi||_{\mathcal{W}^{m,p},g}\leq C.
\]
\end{theorem}

Such generalizations from a single conformal class to a compact set of conformal classes have generally been understood to be folklore, see for instance \cite[discussion following Equation 2.1]{and}, \cite[Lemma 10.1]{kms}.  The point to be checked is the dependence of the constant $C$ in some standard elliptic estimates on the background Riemannian metric $g$:  the $L^p$ estimates $||u||_{\mathcal{W}^{m+2,p},g}\leq C[||\Delta_g u||_{\mathcal{W}^{m,p},g}+||u||_{p,g}]$
(where $C$ can be controlled in terms of the injectivity radius and $||\Ric||_{\mathcal{C}^m,g}$) and the Harnack inequality
$\sup_Mu \leq C \inf_M u$
(where $C$ can be controlled in terms of the diameter, the injectivity radius and a lower Ricci bound); in each case one works first locally in harmonic co-ordinates \cite{hh,hebey} and then globally.

\subsection{A compactness theorem for Yamabe metrics}

\begin{lemma}\label{modulus-upper}
$I$ is upper semi-continuous.
\end{lemma}

\begin{proof}
It is the infimum of a continuous functional.
\end{proof}

In fact $I$ is continuous \cite{BB83}, but we will not need that here.

\begin{proposition} \label{compactness}
Let $m\geq 3$.  Let $(M,c)$ be a conformal manifold with $I(c)<\Lambda$.  Let $(c_k)$ be a sequence of smooth conformal classes on $M$ which $\mathcal{C}^m$-converges to $c$, and let $(g_k)$ be volume-1 Yamabe metrics for the classes $(c_k)$.

Then there exists a subsequence $(g_{k_i})$ which $\mathcal{C}^{m-1}$-converges to a volume-1 Yamabe metric for $c$.
\end{proposition}

\begin{proof}
Choose $p>n$, and choose representatives $(\overline{g}_k)$, $\overline{g}$ of the classes $(c_k)$, $c$ with $\overline{g}_k\to \overline{g}$ in the $\mathcal{C}^m$ topology.  For sufficiently large $k$,
\begin{enumerate}
\item $I(c_k)\leq \tfrac{1}{2}[I(c)+\Lambda]$ (by Lemma \ref{modulus-upper});
\item there are uniform bound on the expressions
\[
\diam(M,\overline{g}_k)\, \quad \injrad(M,\overline{g}_k), \quad ||\Ric(\overline{g}_k)||_{\mathcal{C}^{m-2},\overline{g}_k};
\]
\item there exists a uniform $C$ such that for $k$ sufficiently large,
\[
C^{-1}||\cdot||_{\mathcal{W}^{m+1,p},\overline{g}_k}\leq ||\cdot||_{\mathcal{W}^{m+1,p},\overline{g}} \leq C||\cdot||_{\mathcal{W}^{m+1,p},\overline{g}_k}.
\]
\end{enumerate}

Therefore, by Theorem \ref{higher-order}, if $\varphi_k$ are smooth positive functions on $M$ such that $\varphi_k^{4/(n-2)}\overline{g}_k$ are volume-1 Yamabe metrics for $c_k$, then   we have the uniform bound $||\varphi_k||_{\mathcal{W}^{m,p},\overline{g}}\leq C$, hence by Morrey's inequality the uniform bound $||\varphi_k||_{\mathcal{C}^{m-1,1-\frac{n}{p}},\overline{g}}\leq C$,
and so there exists a subsequence $(k_i)$ such that $\varphi_{k_i}$ $\mathcal{C}^{m-1}$-converges.

It remains to be checked that the limit, $\varphi$, makes $\varphi^{4/(n-2)}\overline{g}$ a Yamabe metric for $c$.  Indeed,  $\varphi_k^{4/(n-2)}\overline{g}_k$ all have volume 1 and constant scalar curvature $I(c_k)$.  So their $\mathcal{C}^{m-1}$-limit $g:=\varphi^{4/(n-2)}\overline{g}$ has volume 1, and constant scalar curvature $\lim_{k\to\infty}I(c_k)$, which is $\leq I(c)$ by Lemma \ref{modulus-upper}; this is a contradiction unless equality holds and $\overline{g}$ is a Yamabe metric.
\end{proof}

\subsection{Notation}

As sketched in the introduction, we denote by  $|\Lambda^nM|$ the bundle of densities on $M$; it is a line bundle, equipped with a natural positive orientation, and hence all real powers $|\Lambda^nM|^\alpha$ are well-defined.  We write $|\Omega|$ for the density associated to an $n$-form $\Omega$.

For each $\alpha$ there is a well-defined ``determinant" map of bundle sections,
\[
\det:\mathcal{C}^\infty(M, \Sym^2(T^*M)\otimes |\Lambda^nM|^\alpha)
\to \mathcal{C}^\infty(M, |\Lambda^nM|^{2+n\alpha}).
\]
For $\alpha=0$ this is the square of the volume form, $g\mapsto (\dVol_g)^{\otimes 2}$.
The case $\alpha=-2/n$ (that is, the bundle $\Sym^2(T^*M)\otimes |\Lambda^nM|^{-2/n}$) has the special feature that its determinant has range $\mathcal{C}^\infty(M, \mathbb{R})$.

Denote by $\mathcal{M}$ the space of smooth Riemannian metrics on $M$, and by $\mathcal{V}$ the space of smooth positive densities.  The map from $\mathcal{M}$  to $\mathcal{V}\times\mathcal{C}^\infty(M, \Sym^2(T^*M)\otimes |\Lambda^nM|^{-2/n})$ given by
\[
g\mapsto (|\dVol_g|, g\otimes|\dVol_g|^{-2/n})
\]
is injective with image $\mathcal{V}\times\mathcal{C}$, where $\mathcal{C}$ denotes the set of smooth positive-definite sections of $\Sym^2(T^*M)\otimes |\Lambda^nM|^{-2/n}$ with determinant 1.  The inverse, an isomorphism from  $\mathcal{V}\times\mathcal{C}$ to $\mathcal{M}$, is given by
$(\Omega, c)\mapsto \Omega^{2/n}c$.

For any element $c\in\mathcal{C}$, the set of Riemannian metrics in $\mathcal{M}$ corresponding to $\mathcal{V}\times\{c\}$ under this isomorphism is precisely a conformal equivalence class of Riemannian metrics.  We therefore identify $c$ with that conformal equivalence class, and $\mathcal{C}$ with the space of smooth conformal classes.  The tangent space $T_c\mathcal{C}$ to $\mathcal{C}$ at $c$ is the kernel of the trace
\[
\tr_c:\mathcal{C}^\infty(M,\Sym^2(T^*M)\otimes |\Lambda^nM|^{-2/n})\to\mathcal{C}^\infty(M,\mathbb{R}).
\]

\subsection{Continuity properties of the total scalar curvature functional}

 We introduce the notation $Q:\mathcal{M}\to \mathbb{R}$ for the total scalar curvature functional,
 \[
 Q(g)=\frac{\int_M R(g)|\dVol_g|}
 {\left(\int_M \dVol_g\right)^{1-2/n}},
 \]
 and, according to the isomorphism described in the previous subsection, have an equivalent functional $Q:\mathcal{V}\times\mathcal{C}\to \mathbb{R}$,
 \[
 Q_\Omega(c)=Q(\Omega^{2/n}c).
 \]

 We have, for a conformal class $c$,
 \[
 I(c)=\inf_{g\in c}Q(g)=\inf_{\Omega\in\mathcal{V}}Q_\Omega(c).
 \]

 \begin{lemma} \label{derivs}
 Let $\Omega$ be a positive density.  The functional $Q_\Omega:\mathcal{C}\to\mathbb{R}$ is $\mathcal{C}^1$, and its derivative $D_c(Q_\Omega):T_c\mathcal{C}\to\mathbb{R}$ at $c\in\mathcal{C}$ is
 \[
 D_c(Q_\Omega)(w)=\Vol(\Omega)^{\frac{2}{n}-1}\int_M\langle w, - \Ric(\Omega^{2/n}c)\Omega^{1-2/n}\rangle_c.
 \]
 \end{lemma}

\begin{proof}
It is well-known (e.g.\ \cite[Proposition 4.17]{Besse}) that for a Riemannian metric $g$ and symmetric 2-tensor $h$,
\[
D_g(Q)(h)=\Vol(g)^{2/n-1}
\int_M \left\langle h, -\Ric(g)+
\tfrac{1}{2}\left[R(g)-\frac{\int_MR(g)\dVol_g}{\Vol(g)}\right]g\right\rangle_g|\dVol_g|,
\]
Recall from the previous subsection that $T_c\mathcal{C}$ is the set of $c$-tracefree sections of $\mathcal{C}^\infty(M,\Sym^2(T^*M))$.  If $w$ is such a section, the tangent vector $(0,w)$ to $\mathcal{V}\times\mathcal{C}$ at $(\Omega,c)$ corresponds to the tangent vector $\Omega^{2/n}w$ to $\mathcal{M}$ at $\Omega^{2/n}c$, and since $\tr_cw=0$, the term $\langle \Omega^{2/n}w,\Omega^{2/n}c\rangle_{\Omega^{2/n}c}$ vanishes.
\end{proof}

\begin{proposition}[Modulus of continuity of $I$] \label{modulus-lower}
Let $(M,c_0)$ be a conformal manifold. Let $v$ be a smooth section of $\Sym^2(T^*M)\otimes |\Lambda^nM|^{-2/n}$, sufficiently small that $\det(c_0+tv)$ is nonvanishing for $t\in[0,1]$.  Write
\[
c_t:=\frac{c_0+tv}{\det(c_0+tv)^{\frac{1}{n}}},
\]
so that $(c_t)$ is a 1-parameter family of conformal classes starting at $c_0$. Let $\Omega$ be a density such that $g_1=\Omega^{2/n}c_1$ is a Yamabe metric for $c_1$, and write $g_t=\Omega^{2/n}c_t$.  Then
\begin{multline*}
\Vol(\Omega)^{1-\frac{2}{n}}\left[I(c_1)-I(c_0)\right]\\
\geq\int_0^1\int_M\left\langle v,
\left(- \Ric(g_t) +\tfrac{1}{n}R(g_t)g_t\right)\det(c_0+tv)^{-1/n}
\Omega^{1-2/n}\right\rangle_{c_t}dt.
\end{multline*}
\end{proposition}

 \begin{proof}[Proof]
Since $g$ is a Yamabe metric for $c_1$,
 \[
 Q_{|\dVol_g|}(c_1)=Q(g)=I(c_1),
 \]
 so
 \begin{eqnarray*}
 I(c_1)-I(c_0)&=&Q_{|\dVol_g|}(c_1)-\left(\inf_{\Omega\in\mathcal{V}}Q_\Omega(c_0)\right)\\
 &\geq&Q_{|\dVol_g|}(c_1)-Q_{|\dVol_g|}(c_0).
 \end{eqnarray*}

 We calculate
 \begin{eqnarray*}
 \frac{dc_{\tau}}{d\tau}\biggr\vert_{\tau=t}
&=&\frac{d}{d\tau}\biggr\vert_{\tau=t}\left[\frac{c_0+\tau v}{\det(c_0+\tau v)^{\frac{1}{n}}}\right]\\
&=&\frac{v\cdot\det(c_0+tv)^{\frac{1}{n}}-(c_0+tv)\cdot\tfrac{1}{n}\det(c_0+tv)^{\frac{1}{n}-1}\cdot \tr_{c_0+tv}v\det(c_0+tv)}
{\det(c_0+tv)^{\frac{2}{n}}}\\
&=&\det(c_0+tv)^{-\frac{1}{n}}\left[v-\tfrac{1}{n}(\tr_{c_t}v)c_t\right],
 \end{eqnarray*}
so by Lemma \ref{derivs},
\begin{eqnarray*}
&&\Vol(\Omega)^{1-\frac{2}{n}}\frac{d}{d\tau}\biggr\vert_{\tau=t}\left[Q_{|\dVol_g|}(c_\tau)\right]\\
&=&
\int_M
\langle v-\tfrac{1}{n}(\tr_{c_t}v)c_t, - \Ric(g_t)\det(c_0+tv)^{-\frac{1}{n}}\Omega^{1-2/n}\rangle_{c_t}\\
&=&
\int_M
\langle v, (- \Ric(g_t)+\tfrac{1}{n}R(g_t)g_t)\det(c_0+tv)^{-\frac{1}{n}}\Omega^{1-2/n}\rangle_{c_t}.
\end{eqnarray*}
The result follows by the Fundamental Theorem of Calculus.
 \end{proof}

\section{An ``Euler-Lagrange inequality''} \label{sect-euler-lagrange}

\renewcommand{\labelenumi}{(\roman{enumi})}

Let $M$ be a compact connected smooth manifold, and suppose the conformal class $c$ on $M$ attains the Yamabe invariant of the manifold $M$.

\begin{proposition} \label{euler-lagrange}
For each distributional section $v$ of $\Sym^2(T^*M)\otimes |\Lambda^nM|^{-2/n}$, there exists a Yamabe metric $g$ for $c$, such that
\[
\int_M \left\langle v, \left(\Ric(g)-\tfrac{1}{n}R(g)g\right) |\dVol_g|^{1-2/n} \right\rangle_c \geq 0.
\]
\end{proposition}

\begin{proof}
If $(M,c)$ is the conformally round sphere, then $c$ contains an Einstein metric $g$, the round metric; for this $g$, the tensor $\Ric(g)-\tfrac{1}{n}R(g)g$ vanishes.  The result follows.

If $(M,c)$ is not the conformally round sphere, by the solution of the Yamabe problem \cite{aub76b,schoen84}, $I(c)<\Lambda$, so the compactness result Proposition \ref{compactness} applies.  We now give a proof in these cases using it.

Let $(v_k)$ be a sequence of smooth $c$-tracefree sections of $\Sym^2(T^*M)\otimes |\Lambda^nM|^{-2/n}$ which converge distributionally to $v$, and let $(t_k)$ be a sequence of positive reals such that for all $m$ we have $t_k||v_k||_{\mathcal{C}^{m},g}\to 0$.  (For instance, one might choose $t_k:= \tfrac{1}{k}||v_k||_{\mathcal{C}^{m},g}^{-1}$.)
Thus the sequence
\[
c_k:=\frac{c+t_kv_k}{\det(c+t_kv_k)^{\frac{1}{n}}},
\]
of smooth conformal classes $\mathcal{C}^\infty$-converges to $c$.  Let $(\Omega_k)$ be densities such that the metrics $\Omega_k^{2/n}c_k$ are volume-1 Yamabe metrics for the classes $(c_k)$.

Since $c$ attains the Yamabe invariant, we have that for each $k$,
\[
0\geq I(c_k)-I(c).
\]
Thus, applying Proposition \ref{modulus-lower} to $c$ and $c_k$,
\begin{eqnarray*}
0&\geq& \frac{I(c_k)-I(c)}{t_k}\\
&\geq&\int_0^1\int_M\left\langle v_k,
\left(- \Ric(g_{k,\tau}) +\tfrac{1}{n}R(g_{k,\tau})g_{k,\tau}\right)\det(c_0+\tau t_k v)^{-1/n}
\Omega_k^{1-2/n}\right\rangle_{c_{k,\tau}}d\tau,
\end{eqnarray*}
where we define the 1-parameter families of conformal classes $c_{k, \tau}$ by
\[
c_{k, \tau}:=\frac{c+\tau t_kv_k}{\det(c+\tau t_kv_k)^{\frac{1}{n}}}
\]
and metrics $g_{k, \tau}$ by $\Omega_k^{2/n}c_{k, \tau}$.

By the compactness result Proposition \ref{compactness} and a diagonal argument, we may choose a subsequence $({k_i})$ such that $g_{k_i}$ converges in $\mathcal{C}^\infty$ to a volume-1 Yamabe metric, $g$ say, for $c$, with $\Omega_{k_i}$ converging in $\mathcal{C}^\infty$ to $|\dVol_g|$.  Therefore also the sequence of fields on $[0,1]\times M$,
\[
(x, \tau)\mapsto
\left\langle \cdot,
\left(- \Ric(g_{k_i,\tau}) +\tfrac{1}{n}R(g_{k_i,\tau})g_{k_i,\tau}\right)\det(c_0+\tau t_{k_i} v)^{-1/n}
\Omega_{k_i}^{1-2/n}\right\rangle_{c_{k_i,\tau}},
\]
$\mathcal{C}^\infty$-converges to the constant-in-$\tau$ field
\[
(x, \tau)\mapsto
\left\langle\cdot,\left(- \Ric(g) +\tfrac{1}{n}R(g)g\right)|\dVol_g|^{1-2/n}\right\rangle_c.
\]
By construction, the sequence $(v_k)$ of constant-in-$\tau$ fields on $[0,1]\times M$ converges distributionally to the constant-in-$\tau$ field $v$.  Pairing, it follows that
\begin{eqnarray*}
&&\int_M \left\langle v, \left(- \Ric(g) +\tfrac{1}{n}R(g)g\right)|\dVol_g|^{1-2/n} \right\rangle_c \\
&=&\lim_{i\to\infty}
\int_0^1\int_M\left\langle v_{k_i},
\left(- \Ric(g_{k_i,\tau}) +\tfrac{1}{n}R(g_{k_i,\tau})g_{k_i,\tau}\right)\det(c_0+\tau t_{k_i} v)^{-1/n}
\Omega_{k_i}^{1-2/n}\right\rangle_{c_{k_i,\tau}}d\tau\\
&\leq& 0.
\end{eqnarray*}
\end{proof}

\section{Hahn-Banach argument} \label{sect-hahn-banach}

\subsection{Preliminaries from functional analysis} \label{functional}

\begin{theorem}[Hahn-Banach separation theorem {\cite[Theorem 3.4]{rudin}}, {\cite[Corollary 2.2.3]{edwards}}] \label{hahn-banach-thm}
Let $X$ be a locally convex topological vector space over $\mathbb{R}$, and $A$ and $B$ nonempty closed convex subsets of $X$, with $B$ compact.  If for all functionals $\varphi\in X^*$ we have
\[
\sup_{a\in A}\varphi(a) \geq \inf_{b\in B}\varphi(b),
\]
then $A\cap B\neq\emptyset$.
\end{theorem}

%{\cite[Theorem IX.2]{diestel}},
\begin{theorem}[{\cite[Proposition 1.2]{phelps}},  {\cite[Corollary 8.13.3]{edwards}}] \label{choquet}
Let $Y$ be a Hausdorff locally convex topological vector space over $\mathbb{R}$, and $K\subseteq Y$ a nonempty compact subset.  If a point $y_0\in Y$ is in the closed convex hull of $K$, then there exists a Borel probability measure $\mu$ on $K$, such that
\[
y_0 = \int_Ky \ d\mu(y).
\]
\end{theorem}

\begin{remarks}
\begin{enumerate}
\item
The integral is to be taken as a \emph{Pettis integral}
%\cite[pp.\ 200-1]{osborne}, 
(\cite[Section 11.10]{ab}, {\cite[Section 8.14]{edwards}}). That is, for all functionals $\psi\in Y^*$,
$\psi(y_0) = \int_K\psi(y) \ d\mu(y)$.
\item By a straightforward argument, the converse is also true.
\end{enumerate}
\end{remarks}

\begin{proof}
Since the closed  convex hull of $K$ contains $y_0$, there exist sequences $N_k$ of natural numbers,  $(a^k_\gamma)_{1\leq \gamma\leq N_k}$ of finite sets of positive reals, and $(y^k_\gamma)_{1\leq \gamma\leq N_k}$ of finite sets of elements of $K$, such that $1=\sum_{\gamma=1}^{N_k}a^k_\gamma$ for all $k$ and $y_0=\lim_{k\to\infty}\sum_{\gamma=1}^{N_k} a^k_\gamma y^k_\gamma$.
Define a sequence of discretely-supported probability measures $(\mu_k)$ on $K$ by
$\mu_k:=\sum_{\gamma=1}^{N_k}a^k_\gamma\delta_{y^k_\gamma}$.  Thus for all functionals $\psi\in Y^*$,
$\psi(y_0)=\lim_{k\to\infty}\int_{K} \psi(y) d\mu_k(y)$.

Applying the Banach-Alaoglu theorem to to the set of Radon measures on $K$, there exists a subsequence $(\mu_{k_i})$ which weak-*-converges to some Borel probability measure $\mu$ on $K$.  For all functionals $\psi\in Y^*$, $\psi(y_0)=\lim_{i\to\infty}\int_{K} \psi(y) d\mu_{k_i}(y)=\int_K\psi(y) \ d\mu(y)$.
\end{proof}

\subsection{Proof of Theorem \ref{conv-hull}}

Let $M$ be a compact smooth $n$-manifold, and $c$ a conformal equivalence class on $M$.

\begin{lemma} \label{pairing}
The conformal equivalence class $c$ induces canonical pairings
\[
\langle \cdot, \cdot \rangle: \Gamma(\Sym^2(T^*M)\otimes |\Lambda^nM|^{s}) \times \mathcal{C}^\infty(\Sym^2(T^*M)\otimes |\Lambda^nM|^{1-4/n-s}) \to \mathbb{R},
\]
($\Gamma$ denoting distributional sections) such that the induced maps
\[
\Gamma(\Sym^2(T^*M)\otimes |\Lambda^nM|^{s}) \to (\mathcal{C}^\infty(\Sym^2(T^*M)\otimes |\Lambda^nM|^{1-4/n-s}))^*
\]
are isomorphisms of topological vector spaces.
\end{lemma}

\begin{proof}
Pick a metric $g\in c$.  The induced pairing $\langle v, w\rangle = \int g(v, w) \ |\dVol_g|^{4/n}$
is easily checked to be independent of the choice of $g\in c$.
\end{proof}

Suppose henceforth that the conformal class $c$ attains $Y(M)$.  Denote by $\mathcal{F}\subseteq \mathcal{C}^\infty(\Sym^2(T^*M))$
the set of volume-1 Yamabe metrics in the conformal class $c$.

Define a function
\[
\mathcal{Q}:c\to \mathcal{C}^\infty(\Sym^2(T^*M)\otimes |\Lambda^nM|^{1-2/n})
\]
on the set of metrics in the conformal class $c$, by
\[
\mathcal{Q}(g):=\left(\Ric(g)-\frac{1}{n}R(g)g\right)|\dVol_g|^{1-2/n}.
\]

\begin{proposition} \label{conv-hull-exp}
The closed convex hull of the set $\mathcal{Q}(\mathcal{F})$ contains 0.
\end{proposition}

\begin{proof}
Let  $A$ be the closed convex hull of the set $\mathcal{Q}(\mathcal{F})$.
 The topological vector space of sections
\[
\mathcal{C}^\infty(\Sym^2(T^*M)\otimes |\Lambda^nM|^{1-2/n})
\]
 is Fr\'echet and therefore locally convex.  Thus it suffices to verify the hypothesis of Theorem \ref{hahn-banach-thm}.  Indeed, let $v$ be in
\[
\Gamma(\Sym^2(T^*M)\otimes |\Lambda^nM|^{-2/n}) \cong (\mathcal{C}^\infty(\Sym^2(T^*M)\otimes |\Lambda^nM|^{1-2/n}))^*
\]
(the isomorphism is by Lemma \ref{pairing}).  By Proposition \ref{euler-lagrange}, there exists (rescaling if necessary) a metric $g\in \mathcal{F}$ such that $\langle v, \mathcal{Q}(g)\rangle\geq 0$.
Since $\mathcal{Q}(g)\in A$, it follows that indeed
$\sup_{a\in A}\langle v, a\rangle\geq 0 =\langle v, 0\rangle$.
\end{proof}

The set $\mathcal{F}$, the set of volume-1 Yamabe metrics in the conformal class $c$, is compact by Proposition \ref{compactness}.  So $\mathcal{Q}(\mathcal{F})$ is also compact.

\begin{remark}
An earlier version of this paper featured the statement that the  convex hull of $\mathcal{Q}(\mathcal{F})$  is closed -- and therefore that the convex hull of $\mathcal{Q}(\mathcal{F})$ itself contains 0, with a corresponding alteration in the paper's main theorem.

It was claimed that the closedness of the  convex hull of $\mathcal{Q}(\mathcal{F})$ follows from the compactness of $\mathcal{Q}(\mathcal{F})$. This argument is incorrect; see for example \cite[Example 5.34]{ab}.  We do not know whether the convex hull of $\mathcal{Q}(\mathcal{F})$ must be closed.
\end{remark}

\begin{proposition} There exists a Borel probability measure $\overline{\mu}$ on $\mathcal{Q}(\mathcal{F})$, such that
$0 = \int_{\mathcal{Q}(\mathcal{F})} a \ d\overline{\mu}(a)$.
\end{proposition}

\begin{proof}
Combine Proposition \ref{conv-hull-exp} with Theorem \ref{choquet}.
\end{proof}

Theorem \ref{conv-hull} now follows by choosing an arbitrary Borel measure $\mu$ on $\mathcal{F}$ such that $\mathcal{Q}_*\mu=\overline{\mu}$.

% *************** Bibliography ***************

\bibliographystyle{alpha}
{\small\bibliography{yamabe}{}}

\end{document}